\documentclass[12pt]{amsart}
\usepackage{amssymb}
\newtheorem{theorem}{Theorem}[section]
\newtheorem{lemma}[theorem]{Lemma}

\newtheorem{proposition}[theorem]{Proposition}

\theoremstyle{remark}
\newtheorem{remark}[theorem]{Remark}

\numberwithin{equation}{section}

\newcommand{\C}{\mathbb{C}}
\newcommand{\Z}{\mathbb{Z}}
\newcommand{\Q}{\mathbb{Q}}

\newcommand{\HH}{\mathfrak{H}}
\newcommand{\tr}[1]{%
{}^t\hspace{-0.5mm}#1%
}%
\newcommand{\Spg}{{\operatorname{Sp}}}
\newcommand{\trace}{\operatorname{tr}}
\newcommand{\rank}{\operatorname{rank}}
\begin{document}

\title[Fourier coefficients of Siegel modular forms]{On Fourier coefficients of Siegel modular forms of degree two with respect to congruence subgroups}

\author{Masataka Chida, Hidenori Katsurada, Kohji Matsumoto}

\address{Department of Mathematics, Graduate School of Science, Kyoto University, Kyoto 606-8502, Japan}

\email{chida@math.kyoto-u.ac.jp}

\address{Muroran Institute of Technology, 27-1 Mizumoto, Muroran, 050-8585, Japan}

\email{hidenori@mmm.muroran-it.ac.jp}

\address{Graduate School of Mathematics, Nagoya University, Chikusa-ku, Nagoya 464-8602, Japan}

\email{kohjimat@math.nagoya-u.ac.jp}



\begin{abstract}
We prove a formula of Petersson's type for Fourier coefficients of Siegel cusp 
forms of degree 2 with respect to congruence subgroups, and as a corollary,
show upper bound estimates of individual Fourier coefficient.   The method in
this paper is essentially a generalization of Kitaoka's previous work which studied
the full modular case, but some modification is necessary to obtain estimates
which are sharp with respect to the level aspect.
\end{abstract}

\maketitle


\section{Introduction and the statement of main results}

Let $M_l(R)$ be the set of $l\times l$ matrices whose components belong to a ring
$R$, and define
\begin{align*}
&\Lambda =\{ S \in \operatorname{M}_2 (\Z) \ | \ \tr{S} = S \},\\
&\Lambda^* = \{ S=(s_{ij}) \in \operatorname{M}_2 (\Q ) \ | \ s_{ii}\in \Z , 
\ 2s_{12}=2s_{21} \in \Z  \}.
\end{align*}
Let $\mathbb{H}_2 =\{ Z=X+iY \in M_2(\C ) \, | \, \tr{Z}=Z,Y>0 \}$ be the Siegel upper
half space of degree 2, 
$\Spg(2,\Z)$ be the Siegel modular group of degree 2 which consists of all
non-singular matrices $M\in M_4(\Z)$ satisfying $\tr{M}JM=J$, where
$J=
\begin{pmatrix}                                                                        
& 1_2 \\   
-1_2 & 
\end{pmatrix}   
$.
For a positive integer $N$, define
$$
\Gamma_0^{(2)}(N) = \left\{ \left.
\begin{pmatrix}
A & B \\
C & D
\end{pmatrix} \in \Spg(2,\Z)
\ \right|\  C \equiv 
\begin{pmatrix}
0 & 0 \\
0 & 0
\end{pmatrix}
\,
\textup{mod} \, N \right\}
.$$
Denote by $S_{k}(\Gamma_0^{(2)}(N))$ the space of Siegel cusp forms of weight $k (\geq 1)$ 
with respect to
$\Gamma_0^{(2)}(N)$, and write the Fourier expansion of $f\in S_{k}(\Gamma_0^{(2)}(N))$ as
\begin{align}\label{1-1}
f(Z)=\sum_{0\leq Q \in \Lambda^*}a_f(Q) e(\trace(QZ)),
\end{align}
where $Z\in \mathbb{H}_2$, $e(x)=\exp(2\pi ix)$ and $\trace$ denotes the trace map.
The main purpose of the present paper is to prove a Petersson type formula 
for the above Fourier coefficients $a_f(Q)$.
Such a study was carried out by Kitaoka \cite{Ki} in the full modular case,
in order to obtain an upper bound estimate of Fourier coefficients.
Our motivation is to generalize Kitaoka's result to the case of general
$\Gamma_0(N)$.

First recall the classical elliptic modular case.   Let $S_k(\Gamma_0(N))$ be the space of 
elliptic cusp forms of weight $k$ and level $N$, $f\in S_k(\Gamma_0(N))$, and denote its 
Fourier coefficients by $a_f(n)$.   Then the classical Petersson formula is
\begin{align}\label{1-2}
\lefteqn{\frac{\Gamma(k-1)}{(4\pi\sqrt{mn})^{k-1}}\sum_{f}
   \frac{\overline{a_f(m)}a_f(n)}{\langle f,f\rangle}}\\
&=\delta_{mn}+\frac{2\pi}{(-1)^{k/2}}
  \sum_{\stackrel{c>0}{c\equiv 0\;(\rm{mod} N)}}\frac{1}{c}S(m,n;c)
  J_{k-1}\left(\frac{4\pi\sqrt{mn}}{c}\right),\notag
\end{align}
where $\langle\;,\;\rangle$ denotes the usual Petersson inner product, 
$\delta_{mn}$ is the Kronecker delta, $S(m,n;c)$ is the Kloosterman sum, 
$J_{k-1}(\cdot)$ is the $(k-1)$-th $J$-Bessel function, and the sum on the left-hand 
side runs over an orthogonal basis of $S_k(N)$ (see Theorem 3.6 in \cite{Iwa}).
From this formula, evaluating the sum on the right-hand side, we can show
\begin{align}\label{1-3}
\frac{\Gamma(k-1)}{(4\pi\sqrt{mn})^{k-1}}\sum_{f}                                
   \frac{\overline{a_f(m)}a_f(n)}{\langle f,f\rangle}
=\delta_{mn}+O((m,n)^{1/2}(mn)^{(k-1)/2}N^{1/2-k}d(N)),
\end{align}
where $d(\cdot)$ denotes the divisor function (Duke \cite{Du}, Kamiya \cite{Ka}).

Now return to the Siegel case.   Let $\mathcal{F}_{k,N}$ be a set of orthogonal 
basis of $S_{k}(\Gamma_0^{(2)}(N))$.   For $Q$, $T\in\Lambda^*$, define
$$
\delta(Q,T)=\# \{ U \in \operatorname{GL}(2,\Z ) | UQ\tr{U} =T \}.
$$
In what follows, $Q$ is to be regarded as fixed, and
$\varepsilon$ denotes an arbitrarily small positive number, not
necessarily the same at each occurrence.
The constants implied by Landau's $O$-symbol and Vinogradov's $\ll$ symbol
may depend on $Q$, $\varepsilon$.

Now we state our main results in the present paper.

\begin{theorem}\label{th1-1}
Let $Q$, $T\in\Lambda^*$, both are positive definite.   Then 

{\rm (i)} We have
\begin{align}\label{1-4}
&\pi^{1\slash 2} (4\pi )^{3-2k}\Gamma (k -\frac{3}{2})\Gamma (k-2)       
(\operatorname{det} Q)^{-k+\frac{3}{2}}\sum_{f\in \mathcal{F}_{k,N}}
\frac{\overline{a_f(Q)}a_f(T)}{\langle f,f\rangle}\\
&\qquad\qquad=\delta(Q,T)+E_Q(T,N),\notag
\end{align}
where $E_Q(T,N)$ is the error term, in the sense that it tends to $0$ when
$N\to\infty$.   Moreover the estimates

{\rm (ii)}
$E_Q(T,N)=O(N^{\frac{3}{2}-k}|T|^{k-\frac{3}{2}}+
N^{2-k+\varepsilon}|T|^{\frac{k}{2}-\frac{1}{4}+\varepsilon}
+N^{3-2k+\varepsilon}|T|^{k-1+\varepsilon}),$

{\rm (iii)}
$ E_Q(T,N)=O(N^{-1/2+\varepsilon}|T|^{k/2-1/4+\varepsilon})$

\noindent hold for $k\geq 3$.
\end{theorem}

\begin{remark}\label{rem-1}
{\rm (i)} When there is no $U\in \operatorname{GL}(2,\Z )$ satisfying $UQ\tr{U} =T$, 
obviously $\delta(Q,T)=0$.    Therefore the role of $\delta(Q,T)$ is similar to
the delta symbol in formula \eqref{1-3}.

{\rm (ii)} This result is a generalization of Proposition 3.3 in
Kowalski-Saha-Tsimerman \cite{KST}.
They applied the estimate to show an equidistribution
result for $L$-functions associated to Siegel cusp forms of genus 2
and growing weight $k$.
So it is expected that our result can be used to prove
a similar result for growing level $N$.

\end{remark}

From the above theorem, as we will see in the next section, we can deduce an
upper bound estimate of individual Fourier coefficient.

\begin{theorem}\label{th1-3}
When there is no $U\in \operatorname{GL}(2,\Z )$ satisfying $UQ\tr{U} =T$,
we obtain

{\rm (a)} $a_f(T)=O(N^{\frac{3}{2}-k}|T|^{k-\frac{3}{2}}+                                     
N^{2-k+\varepsilon}|T|^{\frac{k}{2}-\frac{1}{4}+\varepsilon}                           
+N^{3-2k+\varepsilon}|T|^{k-1+\varepsilon}).$

{\rm (b)} $a_f(T)=O(N^{-1/2+\varepsilon}|T|^{k/2-1/4+\varepsilon})$

\noindent for $k\geq 3$.
\end{theorem}

When $N=1$, 
Theorem \ref{th1-3} (b) is exactly Kitaoka's estimate \cite{Ki}.
However, the estimate with respect to $N$ is rather weak in (b).   
This point is supplied by (a), which gives a sharp estimate with respect to $N$.
This (a) corresponds to the error estimate of Duke-Kamiya in \eqref{1-3}.

In the following sections we will give the proof of the above theorems.
Many parts of the proof are rather straightforward generalizations of Kitaoka's
argument in \cite{Ki}, but we describe the details because we have to trace
carefully how is the effect of $N$.    In particular, some modification of
Kitaoka's argument is necessary to obtain estimates which are sharp with respect
to $N$.

\section{Poincar{\'e} series}

For $M=
\begin{pmatrix}                                                                       
A & B \\                                                                             
C & D                                                                                  
\end{pmatrix} \in \Spg(2,\Z)$ and $Z\in \mathbb{H}_2$, we set
$$
j(M,Z) = \operatorname{det}(CZ+D)
$$
and
$$
M\langle Z \rangle =(AZ+B)(CZ+D)^{-1}.
$$
Moreover, we set
$$
\Gamma_1^{(2)} (\infty )=
\left\{ \left.
\begin{pmatrix}
1_2 & S \\
& 1_2
\end{pmatrix}
\ \right|\
S\in \Lambda
\right\}.
$$
For $Q\in \Lambda^*$ with $Q>0$ and positive integers $k$, $N$, we define the 
Poincar{\'e} series $g_N (Z,Q)$ of
weight $k$ with respect to $\Gamma_0^{(2)}(N)$ by
$$
g_N(Z,Q)=\sum_{M\in \Gamma_1^{(2)}(\infty) \backslash \Gamma_0^{(2)}(N)}
e(\operatorname{tr}(Q\cdot M\langle Z\rangle ))j(M,Z)^{-k}.
$$
For $f,g \in S_{k}(\Gamma_0^{(2)}(N))$, we define the (unnormalized) Petersson norm 
of $f$ and $g$ by
$$     
\langle f,g\rangle =\int_{\Gamma_0^{(2)}(N)\backslash \mathbb{H}_2}f(Z) \overline{g(Z)}        
(\operatorname{det} Y)^{k-3} dZ.                                                       
$$
\begin{proposition}\label{prop2-1}
Let $f\in S_{k}(\Gamma_0^{(2)}(N))$.
Then we have
\begin{align}\label{2-1}
\langle g_N (\cdot ,Q), f \rangle =\pi^{1\slash 2} (4\pi )^{3-2k}
\Gamma (k -\frac{3}{2})\Gamma (k-2)
(\operatorname{det} Q)^{-k+\frac{3}{2}}\overline{a_f(Q)},
\end{align}
and consequently,
\begin{align}\label{2-2}
g_N(Z,Q)=\pi^{1\slash 2} (4\pi )^{3-2k}\Gamma (k -\frac{3}{2})\Gamma (k-2)             
(\operatorname{det} Q)^{-k+\frac{3}{2}}\sum_{f\in \mathcal{F}_{k,N}}                   
\frac{\overline{a_f(Q)}f(Z)}{\langle  f,f \rangle}.                                     
\end{align}
\end{proposition}
\begin{proof}This is a direct generalization of a result
in Klingen's book \cite{Kli}, page 90.
We briefly outline the argument.
We follow the argument in pp.76-90 of Klingen \cite{Kli} with replacing
$\Gamma_n$ and $A_n$ by $\Gamma_0^{(2)}(N)$ and $\Gamma_1^{(2)}(\infty)$, respectively.
The starting point, Proposition 1 (\cite[p.76]{Kli}), does not depend on $N$.
Formulas (7), (8) in \cite[p.78]{Kli} are proved by the technique of
decomposing the Siegel half space $H_n$ into copies 
(by the action of $\{\pm 1\}\backslash \Gamma_n$)
of the fundamental domain $F_n$.   The same technique can be applied to our
present situation, with replacing $S_n^k$ by $S_{n,k}(N)$.
In this way, we follow Klingen's argument until Proposition 3
(\cite[p.85]{Kli}).   In the statement of Proposition 3, the series
$G_n^k(z;g_{\nu})$ is defined, but this is again independent of $N$.
(But be careful with the definition of $\Lambda_n$.)
Also $N$ does not appear in the Fourier expansion of $f\in S_{n,k}(N)$.
On the last line of p.87, Klingen defines $A_n$, which differs from our
$\Gamma_1^{(2)}(\infty)$ by the factor 2.   Therefore, $g_n^k(z,t)$ defined on
p.90 of \cite{Kli} differs from our $g_N(Z,T)$ by the factor 2.
All other parts of the proof are the same as in \cite{Kli}. 
\end{proof}

Substituting \eqref{1-1} (with replacing $Q$ by $T$) into the right-hand side of 
\eqref{2-2}, we have
\begin{align}\label{2-3}
\lefteqn{g_N(Z,Q)=\pi^{1\slash 2} (4\pi )^{3-2k}\Gamma (k -\frac{3}{2})\Gamma (k-2) 
(\operatorname{det} Q)^{-k+\frac{3}{2}}}\\
&\qquad\qquad\times\sum_{0\leq T\in\Lambda^*}
\sum_{f\in \mathcal{F}_{k,N}}                   
\frac{\overline{a_f(Q)}a_f(T)}{\langle  f,f \rangle}e(\trace(TZ)).\notag
\end{align}
Therefore, if we write the Fourier expansion of the Poincar{\'e} series as
\begin{align}\label{2-4}
g_N(Z,Q)=\sum_{0\leq T \in \Lambda^*}A_{Q,N}(T)e(\trace (TZ)),                          
\end{align}
we obtain
\begin{align}\label{2-5}
A_{Q,N}(T)=\pi^{1\slash 2} (4\pi )^{3-2k}\Gamma (k -\frac{3}{2})\Gamma (k-2)       
(\operatorname{det} Q)^{-k+\frac{3}{2}}
\sum_{f\in \mathcal{F}_{k,N}}                       
\frac{\overline{a_f(Q)}a_f(T)}{\langle  f,f \rangle}.
\end{align}
Therefore the Fourier coefficient $A_{Q,N}(T)$ can be estimated by our Theorem
\ref{th1-1}.    In particular, when there is no $U\in \operatorname{GL}(2,\Z )$ 
satisfying $UQ\tr{U} =T$, from Theorem \ref{th1-1} (ii), (iii) (with noting
Remark \ref{rem-1}) we find that $A_{Q,N}(T)$ satisfies the estimations stated in
Theorem \ref{th1-3}.   Since any cusp form can be written as a linear combination of
Poincar{\'e} series, we obtain the assertion of Theorem \ref{th1-3}.

On the other hand, \eqref{2-5} implies that, in order to prove Theorem \ref{th1-1},
it is enough to consider $A_{Q,N}(T)$.    

Let $\HH_N$ be a complete system of representatives of
$\Gamma_1^{(2)}(\infty) \backslash \Gamma_0^{(2)}(N)\slash \Gamma_1^{(2)}(\infty)$.
For $M\in \Spg(2,\Z)$, we denote
$$
\theta (M)=
\left\{
S\in \Lambda \ \left| \ M
\begin{pmatrix}
1_2 & S \\
& 1_2
\end{pmatrix}
M^{-1} \in \Gamma_1^{(2)}(\infty)
\right.
\right\} .
$$
\begin{lemma}[Kitaoka \cite{Ki}, p.158, Lemma 1]\label{lem2-2}
For $M \in \Spg(2,\Z)$, we have
$$
\Gamma_1^{(2)}(\infty)M\Gamma_1^{(2)}(\infty)=\coprod_{S\in \Lambda \slash \theta (M)}\Gamma_1^{(2)}(\infty)M
\begin{pmatrix}
1_2 & S \\
& 1_2
\end{pmatrix}.
$$
\end{lemma}
From this lemma, it is easy to see that
\begin{align}\label{2-6}
g_N(Z,Q)=\sum_{M\in \HH_N}H_Q(M,S),
\end{align}
where
\begin{align}\label{2-7}
H_Q (M,Z)=\sum_{S\in \Lambda \slash \theta (M)}
e(\operatorname{tr}(Q\cdot M\langle Z+S \rangle ))j(M,Z+S)^{-k}.
\end{align}
Write the Fourier expansion of $H_Q (M,Z)$ as
\begin{align}\label{2-8}
H_Q (M,Z)=\sum_{0\leq T \in \Lambda^*}h_Q (M,T)e(\trace (TZ)).
\end{align}
Then
\begin{align}\label{2-9}
h_Q (M,T)=\int_{X (\operatorname{mod}\;1)}H_Q(M,Z)e(-\trace (TZ))dX,
\end{align}
where
$X=
\begin{pmatrix}
x_1 & x_2 \\
x_2 & x_4
\end{pmatrix}
=\Re(Z)$, $dX=dx_1 dx_2 dx_4$.

Comparing \eqref{2-6}, \eqref{2-8} with \eqref{2-4}, we obtain
\begin{align}\label{2-10}
A_{Q,N}(T)=\sum_{M\in \HH_N}h_Q(M,T).
\end{align}
Therefore, to prove Theorem \ref{th1-1}, our remaining task is to evaluate each term
on the right-hand side of \eqref{2-10}.
Let
\begin{align*}
\HH_N^{(i)}=\{ M=\begin{pmatrix}  
A & B \\                                                                               
C & D                                                                                   
\end{pmatrix}\in \HH_N \ |\ \rank C =i \} 
\end{align*}
for $i=0,1$ or $2$, and decompose \eqref{2-10} as
\begin{align}\label{2-11}
A_{Q,N}(T)=\Sigma_0+\Sigma_1+\Sigma_2,
\end{align}
where
$$
  \Sigma_i=\sum_{M\in \HH_N^{(i)}}h_Q(M,T)\qquad(i=0,1,2).
$$

In the following three sections we evaluate $\Sigma_0$, $\Sigma_1$ and
$\Sigma_2$, respectively.

\section{The case of $\rank C=0$}
In this section, we assume $\operatorname{rank}C=0$, i.e. $C=0$.
\begin{lemma}[Kitaoka \cite{Ki}, p.158, Lemma 3]\label{lem3-1}
As
$\HH_N^{(0)}$
we can choose
$$
\left\{ \left.
\begin{pmatrix}
\tr{U} & 0 \\
0 & {U}^{-1}
\end{pmatrix}
\ \right|\
U \in \operatorname{GL}(2,\Z )
\right\}
$$
and $\theta (M) =\Lambda$.
\end{lemma}
\begin{proposition}\label{prop3-2}
We have
$$
\Sigma_0=
\# \{ U \in \operatorname{GL}(2,\Z )\; | \;UQ\tr{U} =T \},
$$
which is hence non-zero only if $Q\sim T$.
\end{proposition}
\begin{proof}
We can choose $M$ which is of the form stated in Lemma \ref{lem3-1}, and
$\theta (M)=\Lambda$. Hence from \eqref{2-7} we have
$$
H_Q(M,Z)=e(\trace (Q\cdot M\langle Z\rangle ))j(M,Z)^k=e(\trace (Q\cdot \tr{U}ZU)).
$$
Therefore \eqref{2-9} gives
\begin{align*}
h_Q(M,T)
=\int_{X \operatorname{mod} 1}e(\trace (Q\cdot \tr{U}ZU)) e(-\trace (TZ))dX.
\end{align*}
Then we have
\begin{align*}
\sum_{M\in \HH_N^{(0)}} h_Q(M,T)
&=\sum_{U\in \operatorname{GL}(2,\Z)}
\int_{X \operatorname{mod} 1}e(\trace (Q\cdot \tr{U}ZU)) e(-\trace (TZ))dX\\
&=\sum_{U\in \operatorname{GL}(2,\Z)}
\int_0^1 \int_0^1 \int_0^1 e(\trace (Q\cdot \tr{U}ZU-TZ))dx_1dx_2dx_4,
\end{align*}
where
$$
X=
\begin{pmatrix}
x_1 & x_2 \\
x_2 & x_4
\end{pmatrix}
=\Re(Z).
$$
Since $\trace (Q\tr{U}ZU-TZ)=0$ is equivalent to $\trace ((UQ\tr{U}-T)Z)=0$,
we see that if $UQ\tr{U}=T$ then $\trace (Q\tr{U}ZU-TZ)=0$ for all $Z$, hence the
above integral is equal to 1.   On the other hand, if $UQ\tr{U}\neq T$, then
$\trace (Q\tr{U}ZU-TZ)\neq 0$ for almost all $Z$, so the above integral vanishes.
This completes the proof.
\end{proof}
\section{The case of $\rank C=1$}

Next we consider the case $\rank C=1$.

\begin{lemma}\label{lem4-1}
As $\HH_N^{(1)}$ we can choose
$$
\left\{ \left.
M=\begin{pmatrix}
* & * \\
U^{-1}C'\tr{V} & U^{-1}D'V^{-1}
\end{pmatrix}
\in \Spg(2,\Z )
\ \right| \
\begin{array}{ll}
U\in G_1, V\in G_2, c_1 \geq 1, d_4=\pm 1,\\ 
(c_1,d_1)=1, d_1, d_2 \operatorname{mod} c_1
\end{array}
\right\}
,$$
where
$$C'=
\begin{pmatrix}
N c_1 & 0 \\
0 & 0
\end{pmatrix}, 
D'=
\begin{pmatrix}
d_1 & d_2 \\
0 & d_4
\end{pmatrix},
G_1=\left\{ 
\begin{pmatrix}
* & * \\
0 & *
\end{pmatrix}
\right\}
\slash\operatorname{GL}(2,\Z),
G_2=\operatorname{GL}(2,\Z)
\left\{
\begin{pmatrix}
1 & * \\
0 & *
\end{pmatrix}
\right\}
.$$
Moreover $\theta (M)$ is given by
$$
\left\{
S\in \Lambda \ \left| \ S[V]=
\begin{pmatrix}
0 & 0 \\
0 & *
\end{pmatrix}
\right.
\right\}
$$
for the above specialized $M$,
where $S[V]=\tr{V}SV$.
\end{lemma}

When $N=1$, this is Kitaoka's Lemma 4 (\cite{Ki}, p.159).
The above generalization is obvious.

For $U,V$ in the setting of Lemma \ref{lem4-1}, we set
$$
P=
\begin{pmatrix}
p_1 & p_2\slash 2 \\
p_2 \slash 2 & p_4
\end{pmatrix}
=Q[\tr{U}]=UQ\tr{U}
$$
and
$$
S=
\begin{pmatrix}
s_1 & s_2\slash 2 \\
s_2 \slash 2 & s_4
\end{pmatrix}
=T[\tr{V}^{-1}]=V^{-1}T\tr{V}^{-1}.
$$
We choose an $a_1$ satisfying $a_1 d_1 \equiv 1 \mod c_1$.
The following is Kitaoka's Lemma 1 (\cite{Ki}, p.160) when $N=1$.

\begin{lemma}\label{lem4-2}
With the notation as above, we have
\begin{align*}
h_Q(M,T)&=(-1)^{k\slash 2}\sqrt{2}\pi |Q|^{\frac{3}{4}-\frac{k}{2}}|T|^{\frac{k}{2}-\frac{3}{4}}
\delta_{p_4, s_4}s_4^{-\frac{1}{2}}(N c_1)^{-\frac{3}{2}}\\
&\times e\left(
\frac{a_1s_4d_2^2 -(a_1d_4p_2-s_2)d_2}{N c_1}+\frac{a_1p_1+d_1s_1}{N c_1}
+\frac{d_4p_2s_2}{2N c_1s_4}
\right)
J_{k-\frac{3}{2}}\left( \frac{4\pi \sqrt{|T||Q|}}{N c_1s_4}\right) .
\end{align*}
\end{lemma}

Using this lemma, we evaluate $\Sigma_1$.   First recall
\begin{align}\label{4-1}
\sum_{n\; {\rm mod}\; c}e
\left(
\frac{an^2+bn}{c}
\right)
=O((a,c)^{\frac{1}{2}}c^{\frac{1}{2}}).
\end{align}
This is a well-known estimate on generalized quadratic Gauss sums, but here we
give a sketch of proof.   Denote the left-hand side by $G(a,b,c)$.   When
$(a,c)>1$, then $G(a,b,c)=0$ unless $(a,c)|b$, while in the latter case
$$
  G(a,b,c)=(a,c)G(a/(a,c),b/(a,c),c/(a,c)),
$$
so we may reduce the problem to the case $(a,c)=1$.   When $(a,c)=1$, we write
$c=uv$, where $u$ is a power of 2 and $v$ is odd.   The decomposition
$G(a,b,c)=G(au,b,v)G(av,b,u)$ holds, and $G(au,b,v)$ is explicitly written in the 
form $\eta v^{\frac{1}{2}}$, where $\eta$ is a certain complex number with 
$|\eta|=1$.   Applying Theorem 10.1 of Hua \cite[Chapter 7]{Hu}, we find
$G(av,b,u)=O(u^{\frac{1}{2}})$ (here, the $\varepsilon$-factor in Hua's
statement is not necessary because now $u$ has only one prime divisor).
Therefore $G(a,b,c)=O(c^{\frac{1}{2}})$ as desired.

Using \eqref{4-1}, we find
$$
\sum_{d_2 \, \textup{mod} \, Nc_1}e
\left(
\frac{a_1 s_4 d_2^2-(a_1 d_4 p_2-s_2)d_2}{Nc_1}
\right)
=O((a_1 s_4 , Nc_1)^{\frac{1}{2}}(Nc_1)^{\frac{1}{2}})
=O((s_4, Nc_1)^{\frac{1}{2}}(Nc_1)^{\frac{1}{2}}),
$$
where the last equality follows from the fact $(a_1, Nc_1)=1$.
Therefore
$$
\left| \sum_{d_2 \, \textup{mod}\,  Nc_1} h_Q(M,T)\right|
\ll \delta_{p_4, s_4}|T|^{\frac{k}{2}-\frac{3}{4}}
s_4^{-\frac{1}{2}}(Nc_1)^{-1}(s_4, Nc_1)^{\frac{1}{2}}
\left|
J_{k-\frac{3}{2}}\left( \frac{4\pi \sqrt{|T||Q|}}{Nc_1 s_4} \right)
\right|
$$
(because $Q$ is fixed and so the $Q$-factor is to be included in the implied constant), 
which further implies
\begin{align}\label{4-2}
&\sum_{U\in G_1}
\sum_{{d_1 \, \textup{mod} \, Nc_1}\atop{(d_1,Nc_1)=1, d_4=\pm 1}}
\left| \sum_{d_2 \, \textup{mod} \, Nc_1} h_Q(M,T)\right| \\
&\ll
\sum_{U\in G_1}
\delta_{p_4, s_4}
|T|^{\frac{k}{2}-\frac{3}{4}}
s_4^{-\frac{1}{2}}(s_4, Nc_1)^{\frac{1}{2}}
\left|
J_{k-\frac{3}{2}}\left( \frac{4\pi \sqrt{|T||Q|}}{Nc_1 s_4} \right)
\right|.\notag
\end{align}
Since $G_1$ is parametrized by the second row up to sign, we see that the right-hand
side of \eqref{4-2} is
\begin{align}\label{4-3}
\ll \sum_{u=\binom{u_3}{u_4}}
\delta_{p_4,s_4}|T|^{\frac{k}{2}-\frac{3}{4}}
s_4^{-\frac{1}{2}}(s_4, Nc_1)^{\frac{1}{2}}
\left|
J_{k-\frac{3}{2}}\left( \frac{4\pi \sqrt{|T||Q|}}{Nc_1 s_4} \right)
\right|,
\end{align}
where $u_3$, $u_4$ is determined by 
$U=
\begin{pmatrix}                                                                       
u_1 & u_2 \\                                                                         
u_3 & u_4
\end{pmatrix}
$.
Since $P=Q[\tr{U}]$, we have
\begin{align}\label{PandQ}
  Q[u]=(u_3\; u_4)Q\begin{pmatrix} u_3\\u_4\end{pmatrix}=p_4.
\end{align}
Therefore $\delta_{p_4,s_4}=\delta_{Q[u],s_4}$, but
$\# \{ u  | \ Q[u]=s_4 \} =O(s_4^{\varepsilon})$.
Hence \eqref{4-3} is
\begin{align*}
&\ll
|T|^{\frac{k}{2}-\frac{3}{4}}
s_4^{-\frac{1}{2}+\varepsilon}
(s_4, Nc_1)^{\frac{1}{2}}
\left|
J_{k-\frac{3}{2}}\left( \frac{4\pi \sqrt{|T||Q|}}{Nc_1 s_4} \right)
\right|.
\end{align*}
Therefore
\begin{align}\label{4-5}
\sum_{V\in G_2}
&\sum_{U\in G_1}
\sum_{{d_1 \, \textup{mod} \, Nc_1}\atop {(d_1,Nc_1)=1, d_4=\pm 1}}
\left| \sum_{d_2 \, \textup{mod} \, Nc_1} h_Q(M,T)\right| \\
&\ll\sum_{V\in G_2}
|T|^{\frac{k}{2}-\frac{3}{4}}
s_4^{-\frac{1}{2}+\varepsilon}
(s_4, Nc_1)^{\frac{1}{2}}
\left|
J_{k-\frac{3}{2}}\left( \frac{4\pi \sqrt{|T||Q|}}{Nc_1 s_4} \right)
\right|.\notag
\end{align}

We see easily that
$V$ is parametrized by the first column, and
$s_4
=T\begin{pmatrix}
-v_3 \\
v_1
\end{pmatrix}
$
for
$V=\begin{pmatrix}
v_1 & v_2 \\
v_3 & v_4 
\end{pmatrix}$.
Therefore, setting 
$$
A(m,T)=\# \{ \begin{pmatrix}
v_1 \\
v_2 
\end{pmatrix} | (v_1 ,v_2)=1 ,
T\begin{pmatrix}
v_1 \\
v_2 
\end{pmatrix}
= m \},
$$
we find that the right-hand side of \eqref{4-5} is
\begin{align}\label{4-6}
\ll
|T|^{\frac{k}{2}-\frac{3}{4}}
\sum_{s_4=1}^{\infty}s_4^{-\frac{1}{2}+\varepsilon}(s_4,Nc_1)^{\frac{1}{2}}
A(s_4,T)
\left|
J_{k-\frac{3}{2}}\left( \frac{4\pi \sqrt{|T||Q|}}{Nc_1 s_4} \right)
\right|.
\end{align}

Here we quote the following well-known estimates:
\begin{align}\label{4-7}
J_{k-\frac{3}{2}}(x)= \left\{
\begin{array}{ll}
{\rm (i)} & O(x^{k-\frac{3}{2}}),\\
{\rm (ii)} & O(x^{-\frac{1}{2}})
\end{array}
\right.
\end{align}
for $x>0$ (see Kitaoka \cite{Ki}, p.163, Lemma 2).
Applying \eqref{4-7} (i), we see that \eqref{4-6} is 
%
\begin{align*}
&\ll
|T|^{\frac{k}{2}-\frac{3}{4}}
\sum_{s_4=1}^{\infty}
s_4^{-\frac{1}{2}+\varepsilon }
(s_4,Nc_1)^{\frac{1}{2}}
A(s_4,T)
\left(\frac{4\pi \sqrt{|T||Q|} }{Nc_1 s_4}
\right)^{k-\frac{3}{2}}\\
&\ll
|T|^{k-\frac{3}{2}}
\sum_{s_4=1}^{\infty}A(s_4,T)s_4^{-\frac{1}{2}+\varepsilon }s_4^{\frac{1}{2}}
(Nc_1s_4)^{-k+\frac{3}{2}}.
\end{align*}
Finally, since 
$\sum_{c_1=1}^{\infty}c_1^{-k+\frac{3}{2}}<+\infty$ (if $k>\frac{5}{2}$),
we have
\begin{align*}
\Sigma_1
&=
\sum_{c_1=1}^{\infty}
\sum_{V\in G_2}
\sum_{U\in G_1}
\sum_{{d_1 \, \textup{mod} \, Nc_1}\atop {(d_1,Nc_1)=1, d_4=\pm 1}}
\sum_{d_2 \, \textup{mod} \, Nc_1} h_Q(M,T) \\
&\ll|T|^{k-\frac{3}{2}}
\sum_{c_1=1}^{\infty}\sum_{s_4=1}^{\infty}
A(s_4, T)s_4^{-k+\frac{3}{2}+\varepsilon}
N^{-k+\frac{3}{2}}
c_1^{-k+\frac{3}{2}}\\
&\ll
|T|^{k-\frac{3}{2}}
N^{-k+\frac{3}{2}}
\sum_{s_4=1}^{\infty}
A(s_4, T)s_4^{-k+\frac{3}{2}+\varepsilon}.
\end{align*}
Since $A(s_4,T)=O(s_4^{\varepsilon})$ (independent of $T$),
we now arrive at the following proposition.
\begin{proposition}\label{prop4-3}
If $k\geq 3$, then
$$
\Sigma_1
\ll |T|^{k-\frac{3}{2}}N^{\frac{3}{2}-k}.
$$
\end{proposition}

This proposition is necessary for the proof of assertion (ii) of Theorem
\ref{th1-1}.   To prove assertion (iii) of Theorem \ref{th1-1}, we have to 
modify the above argument, 
using the both estimates of \eqref{4-7}.   That is, to evaluate the Bessel factor
in \eqref{4-6}, now we apply \eqref{4-7} (ii) if
$4\pi\sqrt{|T||Q|}\geq Nc_1 s_4$, and apply (i) if
$4\pi\sqrt{|T||Q|}< Nc_1 s_4$.   Then \eqref{4-6} is
$$
\ll |T|^{\frac{k}{2}-\frac{3}{4}}(S_1+S_2),
$$
where
\begin{align*}
S_1&=\sum_{s_4\leq \tau/Nc_1}s_4^{-\frac{1}{2}+\varepsilon}(s_4,Nc_1)^{\frac{1}{2}}
   A(s_4,T)|T|^{-\frac{1}{4}}(Nc_1 s_4)^{\frac{1}{2}},\\
S_2&=\sum_{s_4>\tau/Nc_1}s_4^{-\frac{1}{2}+\varepsilon}(s_4,Nc_1)^{\frac{1}{2}}
   A(s_4,T)|T|^{\frac{k}{2}-\frac{3}{4}}(Nc_1 s_4)^{-k+\frac{3}{2}},
\end{align*}
with $\tau=4\pi\sqrt{|T||Q|}$.   Therefore
\begin{align}\label{4-8}
\Sigma_1\ll |T|^{\frac{k}{2}-\frac{3}{4}}\sum_{c_1=1}^{\infty}(S_1+S_2).
\end{align}

\begin{lemma}\label{lem4-4} We have

\begin{align}\label{4-9}
\sum_{1\leq Nc_1\leq \tau/s_4}(s_4,Nc_1)^{\frac{1}{2}}(Nc_1)^{\frac{1}{2}}
\ll N^{-\frac{1}{2}+\varepsilon}\left(\frac{\tau}{s_4}\right)^{\frac{3}{2}}
s_4^{\varepsilon},
\end{align}
\begin{align}\label{4-10}
\sum_{Nc_1>\tau/s_4}(s_4,Nc_1)^{\frac{1}{2}}(Nc_1)^{\frac{3}{2}-k}
\ll N^{-\frac{1}{2}+\varepsilon}\left(\frac{\tau}{s_4}\right)^{\frac{5}{2}-k}
s_4^{\varepsilon}\qquad(k\geq 3).
\end{align}
\end{lemma}

\begin{proof}
When $N=1$, this is Kitaoka's Lemma 2 (\cite{Ki}, p.163).   As for \eqref{4-9}, 
first we write $(s_4,Nc_1)=r$ and $Nc_1=rq$ to obtain 
$$
\sum_{1\leq Nc_1\leq \tau/s_4}(s_4,Nc_1)^{\frac{1}{2}}(Nc_1)^{\frac{1}{2}}
=\sum_{r|s_4}r \sum_{q\leq \tau/s_4 r, N|qr}q^{\frac{1}{2}}.
$$
Put $(N,r)=\nu$, and write $N=\nu N'$.   Then $N|qr$ implies $N'|q$, so we can
write $q=N'q'$.   Therefore the above is
\begin{align*}
&=\sum_{r|s_4}r\sum_{q'\leq \tau/s_4 rN'}(N'q')^{\frac{1}{2}}\\
&\ll \sum_{r|s_4}r\frac{1}{N'}\left(\frac{\tau}{s_4 r}\right)^{\frac{3}{2}}\\
&=\left(\frac{\tau}{s_4}\right)^{\frac{3}{2}}\sum_{\nu|N}\frac{1}{N'}
    \sum_{r|s_4,r\equiv 0({\rm mod}\;\nu)}r^{-\frac{1}{2}}\\
&\leq \left(\frac{\tau}{s_4}\right)^{\frac{3}{2}}\sum_{\nu|N}\frac{1}{N'}
    \nu^{-\frac{1}{2}}d(s_4)\\
&= \left(\frac{\tau}{s_4}\right)^{\frac{3}{2}}\frac{1}{N}\sigma_{1/2}(N)d(s_4)
   \ll N^{-\frac{1}{2}+\varepsilon}\left(\frac{\tau}{s_4}\right)^{\frac{3}{2}}
   s_4^{\varepsilon},
\end{align*} 
where $\sigma_{1/2}(N)=\sum_{d|N}d^{1/2}$.
The proof of \eqref{4-10} is similar, but the condition $k\geq 3$ is necessary
in the course of the proof to assure the convergence of a relevant series.
\end{proof}

Using \eqref{4-9} and $A(s_4,T)\ll s_4^{\varepsilon}$, we have
\begin{align*}
\sum_{c_1=1}^{\infty}S_1&\ll |T|^{-\frac{1}{4}}\sum_{s_4\leq \tau/N}
  \sum_{c_1\leq \tau/s_4 N}(s_4,Nc_1)^{\frac{1}{2}}(Nc_1)^{\frac{1}{2}}
  s_4^{\varepsilon}\\
&\ll |T|^{-\frac{1}{4}}N^{-\frac{1}{2}+\varepsilon}\sum_{s_4\leq \tau/N} 
  \left(\frac{\tau}{s_4}\right)^{\frac{3}{2}}s_4^{\varepsilon}\\
&\ll |T|^{\frac{1}{2}}N^{-\frac{1}{2}+\varepsilon}.
\end{align*}
Also, using \eqref{4-10}, 
\begin{align*}
\sum_{c_1=1}^{\infty}S_2&\ll |T|^{\frac{k}{2}-\frac{3}{4}}\sum_{s_4=1}^{\infty}
  s_4^{-k+1+\varepsilon}\sum_{c_1>\tau/s_4 N}(s_4,Nc_1)^{\frac{1}{2}}
  (Nc_1)^{\frac{3}{2}-k}\\
&\ll |T|^{\frac{k}{2}-\frac{3}{4}}N^{-\frac{1}{2}+\varepsilon}
  \sum_{s_4=1}^{\infty}s_4^{-k+1+\varepsilon}\left(\frac{\tau}{s_4}\right)^
  {\frac{5}{2}-k}\\
&\ll |T|^{\frac{1}{2}}N^{-\frac{1}{2}+\varepsilon}.
\end{align*}
Therefore, from \eqref{4-8} we obtain

\begin{proposition}\label{prop4-5}
If $k\geq 3$, then
$$
\Sigma_1\ll |T|^{\frac{k}{2}-\frac{1}{4}}N^{-\frac{1}{2}+\varepsilon}.
$$
\end{proposition}

\section{The case of $\rank C=2$}

The basic fact for the case $\rank C=2$ is the following lemma,
which is Kitaoka's Lemma 5 (\cite{Ki}, p.159) when $N=1$.

\begin{lemma}\label{lem5-1}
As $\HH_N^{(2)}$, we can choose
$$
\left\{ 
\left.
M=
\begin{pmatrix}
* & * \\
NC & D
\end{pmatrix}
\in \Spg(2,\Z )
\ \right| \
|C|\neq 0, D \, \textup{mod} \, NC\Lambda
\right\}
$$
and $\theta (M)=\{0\}$.
\end{lemma}

The condition $\rank C=2$ is equivalent to $|C|\neq 0$.   For the set of such
matrices, Kitaoka proved:

\begin{lemma}[Kitaoka \cite{Ki}, p.164, Lemma 1]\label{lem5-2}
\begin{align*}
\lefteqn{\left\{
C \in M_2 (\Z) \ | \ |C|\neq 0
\right\}}\\
&=\left\{
\left.
U^{-1}{\rm diag}(c_1,c_2)V^{-1}
\ \right| \
U\in \operatorname{GL}(2,\Z),
V\in \operatorname{GL}(2,\Z)/P(c_2/c_1),
0<c_1 |c_2
\right\}
\end{align*}
where
${\rm diag}(c_1,c_2)=
\begin{pmatrix}   
c_1 & 0 \\                                                                             
0 & c_2                                                                               
\end{pmatrix}
$
and
$$
P(n)=
\left\{
\left.
\begin{pmatrix}
a & b \\
c & d
\end{pmatrix}
\in \operatorname{GL}(2,\Z)
\ \right| \
b\equiv 0 \, \textup{mod} \, n
\right\}
.$$
\end{lemma}

The starting point of the argument is another formula of 
Kitaoka, stated in p.166 of \cite{Ki}, which is
\begin{align*}
h_Q(M,T)=&2^{-1}\pi^{-4}
\left( \frac{|T|}{|Q|} \right)^{\frac{k}{2}-\frac{3}{4}}
||NC||^{-\frac{3}{2}}e(\trace (AC^{-1}Q+C^{-1}DT)/N)\\
&\times\int_0^1 \prod_{i=1}^2 J_{k-\frac{3}{2}}(4\pi s_i u)u(1-u^2)^{-\frac{1}{2}}du,
\end{align*}
where $s_1$ and $s_2$ are positive numbers such that $s_1^2$ and $s_2^2$
are the eigenvalues of the matrix $P_0=T\left[\sqrt{ Q[\tr{(NC)}^{-1}]}\right]$.
(The symbol $||NC||$ means simply the absolute value of the determinant $|NC|$.)
Then
\begin{align}\label{5-1}                                                      
\sum_{D \, \textup{mod} \, NC\Lambda}h_Q(M,T)                                          
=&2^{-1}\pi^{-4}\left( \frac{|T|}{|Q|} \right)^{\frac{k}{2}-\frac{3}{4}}               
||NC||^{-\frac{3}{2}}K(Q,T;NC)\\                                                       
&\times\int_0^1 \prod_{i=1}^2 J_{k-\frac{3}{2}}(4\pi s_i u)u(1-u^2)^{-\frac{1}{2}}du,  
\notag                                                                                 
\end{align}
where
$$
K(Q,T;C)=\sum_{D}e(\trace (AC^{-1}Q+C^{-1}DT))
,$$
with $D$ running over
$$
\left\{
D ({\rm mod} C\Lambda) \in M_2(\Z)\left|
\begin{pmatrix}
A & B \\
C & D
\end{pmatrix}
\in \Spg(2,\Z)
\right.
\right\}
.$$
This $K(Q,T;C)$ is a kind of generalized Kloosterman's sum, introduced and studied
by Kitaoka \cite{Ki}.
In particular, Kitaoka proved:
 
\begin{lemma}[Kitaoka \cite{Ki}, p.150, Proposition 1]\label{lem5-3}
Let $C\in M_2(\Z)$ such that $|C|\neq 0$ and
$
C=U^{-1}
\begin{pmatrix}
c_1 & 0 \\
0 & c_2
\end{pmatrix}
V^{-1}
$, $U,V \in \operatorname{GL}(2,\Z)$,
$0<c_1|c_2$. Then for $P,T\in \Lambda^*$ we have
$$
K(P,T;C)=O(c_1^2 c_2^{\frac{1}{2}+\varepsilon}(c_2, t_4)^{\frac{1}{2}}),
$$
where $\varepsilon$ is any positive number and $t_4$ is the $(2,2)$-entry of $T[V]$.
Moreover, $K(P,T;C)=K(T,P;\tr{C})$ holds.
\end{lemma}
By this lemma, we find that the right-hand side of \eqref{5-1} is
\begin{align}\label{5-2}
\ll
|T|^{\frac{k}{2}-\frac{3}{4}}
||NC||^{-\frac{3}{2}}(Nc_1)^2(Nc_2)^{\frac{1}{2}+\varepsilon}(Nc_2,t_4)^{\frac{1}{2}}
\left|
\int_0^1 \prod_{i=1}^2 J_{k-\frac{3}{2}}(4\pi s_i u)u(1-u^2)^{-\frac{1}{2}}du
\right|
\end{align}
(because $Q$ is fixed).

Kitaoka (\cite{Ki}, p.166) showed that
\begin{align}\label{5-3}
\int_0^1 \prod_{i=1}^2 J_{k-\frac{3}{2}}(4\pi s_i u)u(1-u^2)^{-\frac{1}{2}}du
\ll
\begin{cases}
{\rm (a)} & |P|^{\frac{k}{2}-\frac{3}{4}},\\
{\rm (b)} & |P|^{-\frac{1}{4}},\\
{\rm (c)} & |P|^{\frac{k}{2}-\frac{3}{4}}(\trace (P))^{\frac{1-k}{2}},
\end{cases}
\end{align}
where $P=T\cdot Q[\tr{(NC)}^{-1}]$.
Kitaoka stated the above (a) in case $\trace (P)<1$, (b) in case
$\trace (P)<2|P|$, and (c) otherwise (which are sufficient for his purpose), 
but actually the above estimates themselves are valid without such conditions.
This is because the estimates \eqref{4-7} are true for any $x>0$.
In fact, applying \eqref{4-7} (i) to the both Bessel factors of the left-hand side
of \eqref{5-3}, and noting $(s_1 s_2)^2=|P_0|=|P|$, we obtain the estimate (a).   
Applying \eqref{4-7} (ii) to the
both Bessel factors we obtain (b).   Applying \eqref{4-7} (i) to the Bessel factor
with smaller eigenvalue $s_i$, and applying \eqref{4-7} (ii) to the other Bessel
factor, we obtain (c).
 
We first use only (a) and (c) of \eqref{5-3} to obtain an estimate 
which is sharp with respect to $N$.
It is possible to find a suitable $U_1\in \operatorname{GL}(2,\Z)$ for which
$A=T[V \textup{diag}(c_1, c_2 )^{-1}U_1]$
is Mikowski-reduced.   We may write $C$ in Lemma \ref{lem5-2} as 
$$
C=U^{-1}U_1^{-1}\textup{diag}(c_1,c_2)V^{-1}.
$$
Then we have
$$
|P|=|T\cdot Q[\tr{(NC)}^{-1}]|=N^{-4}|Q|\cdot |A|\asymp N^{-4}|A|
$$
and
$$
\trace P=
\trace (T\cdot Q[\tr{(NC)}^{-1}])\asymp \trace (T\cdot 1_2 [\tr{(NC)}^{-1}])
=N^{-2}\trace (A[U]).
$$
Therefore from \eqref{5-2} we have
\begin{align}\label{5-4}
&\sum_{U\in \operatorname{GL}(2,\Z)}
\left|
\sum_{D \, \textup{mod}\, NC\Lambda}h_Q(M,T)
\right| \\
&\ll |T|^{\frac{k}{2}-\frac{3}{4}}
||NC||^{-\frac{3}{2}}(Nc_1)^2(Nc_2)^{\frac{1}{2}+\varepsilon}(Nc_2,t_4)^{\frac{1}{2}}
\left\{
\sum_{U\in \operatorname{GL}(2,\Z), \trace (A[U])\ll 1}
(N^{-4}|A|)^{\frac{k}{2}-\frac{3}{4}}\right. \notag\\
&+\sum_{U\in \operatorname{GL}(2,\Z), \trace (A[U])\ll |A|}
(N^{-4}|A|)^{\frac{k}{2}-\frac{3}{4}}
+\left. \sum_{U\in \operatorname{GL}(2,\Z), \textup{otherwise}}
(N^{-4}|A|)^{\frac{k}{2}-\frac{3}{4}}
(N^{-2}\trace (A[U]))^{\frac{1-k}{2}}
\right\},\notag
\end{align}
where we applied \eqref{5-3} (a) to the first and third sums, and (c) to the
second sum. 

Kitaoka proved (\cite{Ki}, p.167, Lemma 2) that if $A$ is Minkowski-reduced, then
\begin{align}\label{5-5}
&\sum_{U\in \operatorname{GL}(2,\Z), \trace (A[U])\ll 1}|A|^{\frac{k}{2}-\frac{3}{4}}
+\sum_{U\in \operatorname{GL}(2,\Z), \trace (A[U])\ll |A|}|A|^{-\frac{1}{4}}
+\sum_{U\in \operatorname{GL}(2,\Z), \textup{otherwise}} 
   |A|^{\frac{k}{2}-\frac{3}{4}}\trace (A[U])^{\frac{1-k}{2}}\\
&\quad\ll m(A)^{\varepsilon}\operatorname{max}(1, |A|)^{\frac{3-k}{2}+\varepsilon}
|A|^{\frac{k}{2}-\frac{5}{4}-\varepsilon},\notag 
\end{align}
where $m(A)=\operatorname{min}\{ A[x] |   x\in \Z^2 , x\neq (0,0) \}$.
Using \eqref{5-5}, we find that the first and the third sums on the right-hand side
of \eqref{5-4} are
\begin{align*}
&\ll
(N^{-2k+3}+N^{-k+2})
m(A)^{\varepsilon}\operatorname{max}(1, |A|)^{\frac{3-k}{2}+\varepsilon}
|A|^{\frac{k}{2}-\frac{5}{4}-\varepsilon}\\
&\ll N^{-k+2}m(A)^{\varepsilon}\operatorname{max}(1, |A|)^{\frac{3-k}{2}+\varepsilon} 
|A|^{\frac{k}{2}-\frac{5}{4}-\varepsilon}.
\end{align*}
On the other hand, in the proof of the above Lemma 2 of Kitaoka, he proved
(\cite{Ki}, p.168, line 7) the following 
\begin{lemma}\label{lem5-4}
\begin{align*}
\# \{ U\in \operatorname{GL}(2,\Z) \ | \ \trace (A[U])\ll |A| \} 
&\ll
\# \{ U\in \operatorname{GL}(2,\Z) \ | \ \trace (H[U])\ll |H| \} \\
&\ll
c^{\frac{1}{2}}a^{\frac{1}{2}+\varepsilon},
\end{align*}
where we set
$H=
\begin{pmatrix}
a & 0\\
0 & c
\end{pmatrix}
$
for
$A=
\begin{pmatrix}
a & b\\
b & c
\end{pmatrix}
$.
In particular,
$$
\# \{ U\in \operatorname{GL}(2,\Z) \ | \ \trace (A[U])\ll |A| \}  \ll 
|A|^{\frac{1}{2}+\varepsilon}
$$
if $A$ is Minkowski-reduced.
\end{lemma}

Therefore, the second sum on the right-hand side of \eqref{5-4} is
\begin{align*}
\ll
N^{-2k+3} |A|^{\frac{k}{2}-\frac{3}{4}}
\sum_{U\in \operatorname{GL}(2,\Z),\trace (A[U])\ll |A|}1
\ll N^{-2k+3} |A|^{\frac{k}{2}-\frac{1}{4}+\varepsilon}.
\end{align*}

Collecting the above estimates, and noting $||NC||=N^2 |c_1 c_2|$, 
from \eqref{5-4} we obtain
\begin{align}\label{5-6}
&\sum_{U\in \operatorname{GL}(2,\Z)}
\left|
\sum_{D \, \textup{mod}\, NC\Lambda}h_Q(M,T)
\right| 
\ll
|T|^{\frac{k}{2}-\frac{3}{4}}
N^{-\frac{1}{2}+\varepsilon}|c_1|^{\frac{1}{2}}
|c_2|^{-1+\varepsilon}(Nc_2, t_4)^{\frac{1}{2}}\\
&\;\;\times \left\{ N^{-2k+3}|A|^{\frac{k}{2}-\frac{1}{4}+\varepsilon}
+N^{-k+2}m(A)^{\varepsilon}\operatorname{max}(1, |A|)^{\frac{3-k}{2}+\varepsilon}
|A|^{\frac{k}{2}-\frac{5}{4}-\varepsilon}\right\}.\notag
\end{align}
Similarly to \eqref{PandQ}, we see that $t_4=T[v]$, where $v$ is the vector 
consisting of the second column of $V$.
Using this fact, $|A|=|T|(c_1 c_2)^{-2}$ and 
$(Nc_2 , T[v])^{\frac{1}{2}}\leq N^{\frac{1}{2}}(c_2, T[v])^{\frac{1}{2}}$,
we find that the right-hand side of \eqref{5-6} is
\begin{align*}
\ll &
|T|^{\frac{k}{2}-\frac{3}{4}}
N^{\varepsilon}c_1^{\frac{1}{2}}c_2^{-1+\varepsilon}
(c_2,T[v])^{\frac{1}{2}}
\left\{
N^{-2k+3}|T|^{\frac{k}{2}-\frac{1}{4}+\varepsilon}
(c_1c_2)^{-k+\frac{1}{2}+\varepsilon}
\right. \\
&\left. 
+N^{-k+2} m(T[V\textup{diag}(c_1,c_2)^{-1}U_1^{-1}])^{\varepsilon}
\cdot \operatorname{max}(1, |T| (c_1c_2)^{-2})^{\frac{3-k}{2}+\varepsilon}
\cdot |T|^{\frac{k}{2}-\frac{5}{4}-\varepsilon}(c_1c_2)^{-k+\frac{5}{2}+\varepsilon}
\right\}
.\end{align*}
Therefore,
\begin{align}\label{5-7}
\sum_{0<c_1 |c_2} \sum_{V\in \operatorname{GL}(2,\Z) \slash P(c_1\slash c_2)}
\sum_{U\in \operatorname{GL}(2,\Z)}
&\left|
\sum_{D \, \textup{mod}\, NC\Lambda}h_Q(M,T)
\right| \\
&=\sum_{0<c_1 |c_2} \sum_{V\in \operatorname{GL}(2,\Z) \slash P(c_1\slash c_2)}
\left\{ W_1(c_1,c_2,V)+W_2(c_1,c_2,V)\right\}
,\notag\end{align}
where
\begin{align*}
W_1(c_1,c_2,V)
=&|T|^{\frac{k}{2}-\frac{3}{4}}
N^{-2k+3+\varepsilon}c_1^{-k+1+\varepsilon}c_2^{-k-\frac{1}{2}+\varepsilon}
(c_2,T[v])^{\frac{1}{2}}|T|^{\frac{k}{2}-\frac{1}{4}+\varepsilon},\\
W_2(c_1,c_2,V)
=&|T|^{\frac{k}{2}-\frac{3}{4}}
N^{-k+2+\varepsilon}c_1^{-k+3+\varepsilon}c_2^{-k+\frac{3}{2}+\varepsilon}
(c_2,T[v])^{\frac{1}{2}}m(T[V\textup{diag}(c_1,c_2)^{-1}U_1^{-1}])^{\varepsilon}\\
&\times \operatorname{max}(1, |T|(c_1c_2)^{-2})^{\frac{3-k}{2}+\varepsilon}
\cdot |T|^{\frac{k}{2}-\frac{5}{4}-\varepsilon}.
\end{align*}
The form of $W_2(c_1,c_2,V)$ is quite similar to the right-hand side of
Kitaoka's Lemma 3 (\cite{Ki}, p. 169).    Therefore, by the same argument as in
p.170 of Kitaoka \cite{Ki}, we obtain 
\begin{align}\label{5-8}
\sum_{0<c_1 |c_2} \sum_{V\in \operatorname{GL}(2,\Z) \slash P(c_1\slash c_2)}
W_2(c_1,c_2,V)
\ll N^{2-k+\varepsilon}|T|^{\frac{k}{2}-\frac{1}{4}+\varepsilon}.
\end{align}
On the other hand,
\begin{align}\label{5-9}
\lefteqn{\sum_{0<c_1 |c_2} \sum_{V\in \operatorname{GL}(2,\Z) \slash P(c_1\slash c_2)}
W_1(c_1,c_2,V)}\\
&=
|T|^{k-1+\varepsilon}N^{3-2k+\varepsilon}
\sum_{0<c_1 |c_2} c_1^{-k+1+\varepsilon}c_2^{-k-\frac{1}{2}+\varepsilon}
\sum_{V\in \operatorname{GL}(2,\Z) \slash P(c_1\slash c_2)}
(c_2, T[v])^{\frac{1}{2}}.\notag
\end{align}
To evaluate this double sum, here we quote one more result of Kitaoka.
For $G=(g_{ij})\in \Lambda^*$, set
$e(G):=gcd(g_{11}, g_{22}, 2g_{12})$.
Let $S=\left\{ 
\begin{pmatrix}                                                                        
b\\                                                                                    
d                                                                                      
\end{pmatrix}                                                                          
; b,d \in \Z ,(b,d)=1 \right\}$.
For a positive integer $n$, define $S(n)=S\slash \sim$,
where we denote
$                                                                                      
\begin{pmatrix}                                                                        
b\\                                                                                    
d                                                                                      
\end{pmatrix}                                                                          
\sim                                                                                   
\begin{pmatrix}                                                                        
b'\\                                                                                    
d'                                                                                      
\end{pmatrix}                                                                          
$
if there exists a $w\in \Z$ such that $(w,n)=1$ and
$\begin{pmatrix}                                                                       
b\\                                                                                    
d                                                                                      
\end{pmatrix}                                                                          
\equiv w                                                                                     
\begin{pmatrix}                                                                        
b'\\         
d'                                                                                       
\end{pmatrix}                                                                          
({\rm mod}\; n).
$

\begin{lemma}[Kitaoka \cite{Ki}, p.154, Proposition 2]\label{lem5-5}
For any $P\in \Lambda^*$, we have
$$
\sum_{x\in S(n)}(P[x],n)^{\frac{1}{2}}=O(n^{1+\varepsilon}(e(P),n)^{\frac{1}{2}}).
$$
\end{lemma}

Applying this lemma to $P[x]=T[v]$, $n=c_2/c_1$.   Then 
\begin{align*}
\sum_{V\in \operatorname{GL}(2,\Z)\slash P(c_2\slash c_1)}
(c_2, T[v])^{\frac{1}{2}}
&\leq
\sum_{V\in \operatorname{GL}(2,\Z)\slash P(c_2\slash c_1)}
c_1^{\frac{1}{2}}(c_2\slash c_1, T[v])^{\frac{1}{2}}\\
&\ll c_1^{\frac{1}{2}}(c_2\slash c_1)^{1+\varepsilon}
(c_2 \slash c_1,T[v])^{\frac{1}{2}}\\
&\ll c_1^{\frac{1}{2}}(c_2\slash c_1)^{\frac{3}{2}+\varepsilon},
\end{align*}
where the second inequality follows from the lemma and the fact
$e(T)\leq t_4=T[v]$.
Therefore we have
\begin{align*}
&\sum_{0<c_1|c_2}
c_1^{-k+1+\varepsilon}c_2^{-k-\frac{1}{2}+\varepsilon}
\sum_{V\in \operatorname{GL}(2,\Z)\slash P(c_2\slash c_1)}
(c_2,T[v])^{\frac{1}{2}}\\
&\ll
\sum_{0<c_1|c_2}c_1^{-k+\varepsilon}c_2^{-k+1+\varepsilon}
=\sum_{c_1=1}^{\infty}\sum_{c_3=1}^{\infty}c_1^{-2k+1+\varepsilon}
c_3^{-k+1+\varepsilon},
\end{align*}
where we put $c_2=c_1 c_3$.  This is convergent when $k\geq 3$.
Then from \eqref{5-9} we have
\begin{align}\label{5-10}
\sum_{0<c_2|c_1}\sum_{V\in \operatorname{GL}_2(\Z)\slash P(c_2\slash c_1)}
W_1(c_1,c_2,V)\ll
|T|^{k-1+\varepsilon}N^{3-2k+\varepsilon}.
\end{align}
Substituting \eqref{5-8} and \eqref{5-10} into the right-hand side of \eqref{5-7},
we now obtain
\begin{proposition}\label{prop5-6}
If $k\geq 3$, then
$$
\Sigma_2
\ll
N^{2-k+\varepsilon}|T|^{\frac{k}{2}-\frac{1}{4}+\varepsilon}
+N^{3-2k+\varepsilon}|T|^{k-1+\varepsilon}.
$$
\end{proposition}

Next, we make use of \eqref{5-3} (b) as well as (a) and (c) to deduce 
another estimate of $\Sigma_2$.   This time, instead of $A$, we use
$A'=T[V{\rm diag}(Nc_1,Nc_2)^{-1} U_1]$.   Then
$|P|\asymp |A'|$, $\trace P\asymp\trace(A'[U])$, and hence
the sums in the curly parentheses
on the right-hand side of \eqref{5-4} are replaced by
\begin{align*}
&\sum_{U\in \operatorname{GL}(2,\Z), \trace (A'[U])\ll 1}                                
|A'|^{\frac{k}{2}-\frac{3}{4}}                                  
+\sum_{U\in \operatorname{GL}(2,\Z), \trace (A'[U])\ll |A'|}                            
|A'|^{-\frac{1}{4}}\\                                                  
&\quad+\sum_{U\in \operatorname{GL}(2,\Z), \textup{otherwise}}                        
|A'|^{\frac{k}{2}-\frac{3}{4}}                                                  
(\trace (A'[U]))^{\frac{1-k}{2}},  
\end{align*}
for which \eqref{5-5} can be directly applied.   Therefore
\begin{align*}
&\sum_{0<c_1 |c_2} \sum_{V\in \operatorname{GL}(2,\Z) \slash P(c_1\slash c_2)}         
\sum_{U\in \operatorname{GL}(2,\Z)}                                                    
\left|                                                                                
\sum_{D \, \textup{mod}\, NC\Lambda}h_Q(M,T)                                           
\right| \\
&\ll\sum_{0<c_1 |c_2} \sum_{V\in \operatorname{GL}(2,\Z) \slash P(c_1\slash c_2)}
|T|^{\frac{k}{2}-\frac{3}{4}}||NC||^{-\frac{3}{2}}(Nc_1)^2(Nc_2)^{\frac{1}{2}+
\varepsilon}(Nc_2,t_4)^{\frac{1}{2}}\\
&\qquad\qquad\times m(A')^{\varepsilon}{\rm max}(1,|A'|)^{\frac{3-k}{2}+\varepsilon}
|A'|^{\frac{k}{2}-\frac{5}{4}-\varepsilon}\\
&\ll N^{-2k+\frac{9}{2}+\varepsilon}|T|^{k-2+\varepsilon}
\sum_{0<c_1 |c_2} \sum_{V\in \operatorname{GL}(2,\Z) \slash P(c_1\slash c_2)}  
c_1^{-k+3+\varepsilon}c_2^{-k+\frac{3}{2}+\varepsilon}\\
&\qquad\qquad\times (Nc_2,T[v])^{\frac{1}{2}}{\rm max}(1,N^{-4}|T|(c_1 c_2)^{-2})
^{\frac{3-k}{2}+\varepsilon}\\
&=R_1+R_2,
\end{align*}
say, where $R_1$ denotes the part with $(c_1 c_2)^2\geq N^{-4}|T|$, and $R_2$
the remaining part.   Then
\begin{align*}
R_2&=N^{-\frac{3}{2}+\varepsilon}|T|^{\frac{k}{2}-\frac{1}{2}+\varepsilon}
\sum_{{0<c_1 |c_2}\atop{(c_1 c_2)^2< N^{-4}|T|}}c_2^{-\frac{3}{2}+\varepsilon}
\sum_{V\in \operatorname{GL}(2,\Z) \slash P(c_1\slash c_2)}
(Nc_2,T[v])^{\frac{1}{2}}\\
&\ll N^{-\frac{3}{2}+\varepsilon}|T|^{\frac{k}{2}-\frac{1}{2}+\varepsilon}
\sum_{{0<c_1 |c_2}\atop{(c_1 c_2)^2< N^{-4}|T|}}c_2^{-\frac{3}{2}+\varepsilon}
(Nc_1)^{\frac{1}{2}}\sum_{V\in \operatorname{GL}(2,\Z) \slash P(c_1\slash c_2)} 
(c_2/c_1,T[v])^{\frac{1}{2}}\\
&\ll N^{-1+\varepsilon}|T|^{\frac{k}{2}-\frac{1}{2}+\varepsilon}
\sum_{{0<c_1 |c_2}\atop{(c_1 c_2)^2< N^{-4}|T|}}c_1^{\frac{1}{2}}
c_2^{-\frac{3}{2}+\varepsilon}\left(\frac{c_2}{c_1}\right)^{1+\varepsilon}
(c_2/c_1,e(T))^{\frac{1}{2}}
\end{align*}
by Lemma \ref{lem5-5}.   Putting $c_2=nc_1$, we obtain
\begin{align*}
R_2\ll N^{-1+\varepsilon}|T|^{\frac{k}{2}-\frac{1}{2}+\varepsilon}
\sum_{c_1=1}^{\infty}c_1^{-1+\varepsilon}\sum_{n<|T|^{1/2}(c_1 N)^{-2}}
n^{-\frac{1}{2}+\varepsilon}(n,e(T))^{\frac{1}{2}}.
\end{align*}
The last sum is $O(e(T)^{\varepsilon}(|T|^{\frac{1}{2}}(c_1 N)^{-2})^{\frac{1}{2}
+\varepsilon})$ as is shown in p.170 of Kitaoka \cite{Ki}.   Hence
\begin{align}\label{5-11}
R_2\ll N^{-2+\varepsilon}|T|^{\frac{k}{2}-\frac{1}{4}+\varepsilon}
e(T)^{\varepsilon}\sum_{c_1=1}^{\infty}c_1^{-2+\varepsilon}
\ll N^{-2+\varepsilon}|T|^{\frac{k}{2}-\frac{1}{4}+\varepsilon}.
\end{align}

The remaining part $R_1$ can be treated similarly, and it is estimated by exactly
the same right-hand side as that of \eqref{5-11}.   Therefore we now obtain

\begin{proposition}\label{prop5-7}
$$
\Sigma_2\ll N^{-2+\varepsilon}|T|^{\frac{k}{2}-\frac{1}{4}+\varepsilon}.
$$
\end{proposition}

\section{Completion of the proof of Theorem 1.1}

From \eqref{2-11}, Proposition \ref{prop3-2}, Proposition \ref{prop4-3} and 
Proposition \ref{prop5-6},
we have
\begin{align}\label{6-1}
A_{Q,N}(T)
&=\# \{ U\in \operatorname{GL}_2(\Z)\ | \ UQ\tr{U}=T \}\\
&+O(N^{\frac{3}{2}-k}|T|^{k-\frac{3}{2}}+
N^{2-k+\varepsilon}|T|^{\frac{k}{2}-\frac{1}{4}+\varepsilon}
+N^{3-2k+\varepsilon}|T|^{k-1+\varepsilon})\notag
\end{align}
for $k\geq 3$.   This implies assertions (i) and (ii) of Theorem
\ref{th1-1}.    On the other hand, from \eqref{2-11}, Proposition \ref{prop3-2}, 
Proposition \ref{4-5} and proposition \ref{5-7}, assertion (iii) follows.
The proof of Theorem \ref{th1-1} is thus complete.


\end{document}